\documentclass[12pt]{amsart}

\usepackage{amsfonts}
\usepackage{amssymb}
\usepackage[letterpaper, left=2.5cm, right=2.5cm, top=2.5cm,
bottom=2.5cm,dvips]{geometry}
\usepackage{verbatim}
\usepackage{graphicx}
\usepackage{diagbox}
\usepackage{xcolor}

\usepackage[backref]{hyperref}
\hypersetup{
	colorlinks,
	linkcolor={blue!60!black},
	citecolor={green!60!black},
	urlcolor={green!60!black}
}

\newtheorem{theorem}{Theorem}
\theoremstyle{plain}

\newtheorem{conjecture}{Conjecture}
\newtheorem{corollary}{Corollary}

\numberwithin{equation}{section}

\DeclareMathOperator{\Arb}{arb}
\DeclareMathOperator{\arb}{arb}
\DeclareMathOperator{\col}{col}

\begin{document}
\title[Strong arboricity of graphs]{Strong arboricity of graphs}
\author[T. Bartnicki]{Tomasz Bartnicki}
\address{Stanis\l aw Staszic State University of Applied Sciences in Pi\l a, 64-920 Pi\l a, Poland}
\email{t.bartnicki@ans.pila.pl}

\author[S. Czerwiński]{Sebastian Czerwi\'{n}ski}
\address{Faculty of Mathematics, Computer Science, and Econometrics,
University of Zielona G\'{o}ra, 65-516 Zielona G\'{o}ra, Poland}
\email{s.czerwinski@wmie.uz.zgora.pl}

\author[J. Grytczuk]{Jaros\l aw Grytczuk}
\address{Faculty of Mathematics and Information Science, Warsaw University
of Technology, 00-662 Warsaw, Poland}
\email{grytczuk@tcs.uj.edu.pl}
\thanks{The third author was supported in part by Narodowe Centrum Nauki, grant 2020/37/B/ST1/03298.}

\author[Z. Miechowicz]{Zofia Miechowicz}
\address{Stanis\l aw Staszic State University of Applied Sciences in Pi\l a, 64-920 Pi\l a, Poland}
\email{z.miechowicz@ans.pila.pl}

\begin{abstract}
An edge coloring of a graph $G$ is \emph{woody} if no cycle is monochromatic. The \emph{arboricity} of a graph $G$, denoted by $\arb (G)$, is the least number of colors needed for a woody coloring of $G$. A coloring of $G$ is \emph{strongly woody} if after contraction of any single edge it is still woody. In other words, not only any cycle in $G$ can be monochromatic but also any \emph{broken cycle}, i.e., a simple path arising by deleting a single edge from the cycle. The least number of colors in a strongly woody coloring of $G$ is denoted by $\zeta(G)$ and called the \emph{strong arboricity} of $G$.

We prove that $\zeta(G)\leqslant \chi_a(G)$, where $\chi_a(G)$ is the \emph{acyclic chromatic number} of $G$ (the least number of colors in a proper vertex coloring without a $2$-colored cycle). In particular, we get that  $\zeta(G)\leqslant 5$ for planar graphs and $\zeta(G)\leqslant 4$ for otuterplanar graphs. We conjecture that $\zeta(G)\leqslant 4$ holds for all planar graphs. We also prove that $\zeta(G)\leqslant 4(\arb(G))^2$ holds for arbitrary graph $G$. A natural generalziation of strong arboricity to \emph{matroids} is also discussed, with a special focus on cographic matroids.
\end{abstract}

\maketitle

\section{Introduction}
Let $G$ be a simple graph. Consider an edge coloring of $G$ such that every color class forms a forest, i.e., a subgraph not containing any cycles. For convenience, we call such colorings \emph{woody}. The \emph{arboricity} of a graph $G$, denoted by $\arb(G)$, is the least number of colors in a woody coloring of $G$.

By the well-known Nash-Williams' theorem \cite{Nash-Williams}, $\arb(G)=\lceil f(G) \rceil$, where $f(G)$ is the \emph{fractional arboricity} of $G$, a quantity given by:
\begin{equation}
	f(G)=\max_{H\subseteq G}{\frac{|E(H)|}{|V(H)|-1}}.
\end{equation}
In particular, $\Arb(G)\leqslant 3$ for any simple planar graph $G$, since every such graph on $n$ vertices has at most $3n-6$ edges.

An interesting higher order analog of graph arboricity was introduced by Ne\v{s}et\v{r}il, Ossona de Mendez, and Zhu \cite{NesetrilOssonaZhu}. It is defined as follows. For a fixed positive integer $p$, a coloring of the edges of a graph $G$ is called \emph{$p$-woody} if every cycle $C$ in $G$ contains at least $\min\{|C|,p+1\}$ colors. The least number of colors needed for a $p$-woody coloring of $G$ is called the $p$\emph{-arboricity} of a graph $G$, and denoted by $\Arb_{p}(G)$. Thus $\Arb_1(G)=\Arb(G)$ is the standard arboricity, while $\Arb_2(G)$ can be seen as a relaxed version of the \emph{acyclic chromatic index} of $G$, denoted as $a'(G)$. This is a well studied graph invariant introduced independently by Fiam\v{c}ik \cite{Fiamcik}, and Alon, Sudakov, and Zaks \cite{AlonSudakovZaks}, defined as the least number of colors in a \emph{proper} edge coloring with no $2$-colored cycles. In \cite{GreenhilPikhurko}, Greenhill and Pikhurko studied a generalized acyclic chromatic index $a'_p(G)$ defined  by a similar condition for cycles.

In \cite{NesetrilOssonaZhu}, Ne\v{s}et\v{r}il, Ossona de Mendez, and Zhu proved a theorem characterizing classes of graphs with bounded $p$-arboricity in terms of expansion parameters. Further results were obtained in \cite{BartnickiBCFGM}. In particular, it is proved there that $\Arb_{p}(G)\leqslant p+1$ for every planar graph $G$ of girth at least $2^{p+1}$, for all $p\geqslant 2$. And a similar result holds for graphs of any fixed genus. Additionally for outerplanar graphs we have that $\Arb_{2}(G)\leqslant 5$, which is optimal. Notice however, that $2$-arboricity is not bound for planar graphs. Indeed, it is not hard to see that $\arb_2(K_{2,n})$ exceeds any finite upper bound if $n$ is sufficiently large.

In this paper we introduce another variant of graph arboricity, which is located between $\arb_1(G)$ and $\arb_2(G)$. It is defined as follows. A subgraph $B$ of a simple graph $G$ is called a \emph{broken cycle} if $B=C-e$ for some cycle $C$ and some edge $e$ of $C$. A coloring of the edges of $G$ is \emph{strongly woody} if no broken cycle is monochromatic. The least number of colors needed for such a coloring is denoted by $\zeta(G)$ and called the \emph{strong arboricity} of $G$. It is not hard to see that every graph $G$ satisfies $
\arb_1(G)\leqslant \zeta(G)\leqslant\arb_2(G)$.

Equivalently, the strongly woody coloring of $G$ is just a woody coloring that stays woody in a graph $G/e$ obtained by \emph{contraction} of an arbitrary edge $e$ in $G$. This is analogous to the well studied notion of \emph{strong edge coloring} of graphs (see \cite{Molloy-Reed}), where the coloring must be proper even if any single edge is contracted.

We prove a number of results on the strong arboricity relating this new parameter to the existing ones. For instance, we prove that $\zeta(G)\leqslant \chi_a(G)$, where $\chi_a(G)$ is the \emph{acyclic chromatic number} of $G$, defined as the least number of colors in a proper vertex coloring without a $2$-colored cycle. In particular, by the famous theorem of Borodin \cite{Borodin}, we have $\zeta(G)\leqslant 5$ for planar graphs. We conjecture however that $\zeta(G)\leqslant 4$ holds for all planar graphs. We also prove that $\zeta(G)\leqslant 4(\arb(G))^2$ holds for arbitrary graph $G$. We also discuss some natural variants of strong arboricity (list, paint, game, etc.), as well as generalizations of these notions to \emph{matroids}.

\section{The results}

We start with some examples and simple observations. First notice that in a strongly woody coloring every triangle must be rainbow (no color occurs twice). In general, if a cycle $C$ has at least three colors, then no broken cycle $B=C-e$ arising from $C$ can be monochromatic. However, a cycle $C$ (of size bigger than $3$) may be $2$-colored and still can be safe, what happens exactly when each of the two colors occurs at least twice on $C$ (see Fig. \ref{Zlamasy1}).
\begin{figure}[ht]
	
	\begin{center}
		
		\resizebox{13cm}{!}{
			
			\includegraphics{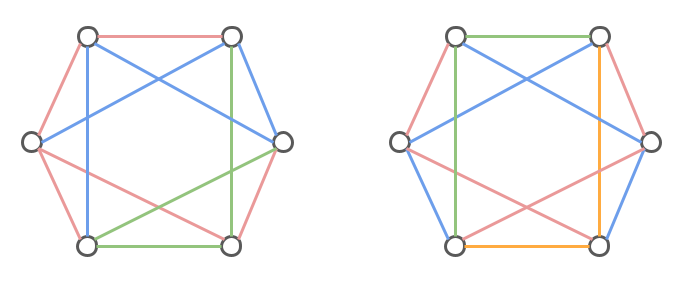}
			
		}
		
		\caption{An example of a graph $G$ with $\arb(G)=3$ and $\zeta(G)=4$ (a woody 3-coloring (left) and a strongly woody 4-coloring (right)).}
		\label{Zlamasy1}
	\end{center}
\end{figure}
\subsection{Planar graphs}
We start with showing that the strong arboricity is at most $\chi_a(G)$. Recall that $\chi_a(G)$ is the least number of colors in a proper vertex coloring of $G$, in which no cycle is $2$-colored. So, there is no problem with odd cycles as they must be $3$-colored, but one has to place the third color also on every even cycle.

\begin{theorem}\label{Theorem Acyclic}
Every simple graph $G$ satisfies $\zeta(G)\leqslant \chi_a(G)$.
\end{theorem}
\begin{proof}Assume that $\chi_a(G)=k$ and fix any acyclic coloring $f:V(G)\rightarrow \mathbb{Z}_k$. Consider a derived coloring of the edges of $G$ defined for any edge $e=uv$ by $g(e)=f(u)+f(v)$ in the group $\mathbb{Z}_k$. We claim that in this coloring no broken cycle is monochromatic. Indeed, suppose that $C-e$ is monochromatic for some cycle $C$. Denote the vertices of $C$ in the cyclic order as $v_1,v_2,\ldots,v_r$, with $r\geqslant3$. Assume that $e=v_1v_r$. Since $C-e$ is monochromatic, we have
	\begin{equation}
	g(v_1v_2)=g(v_2v_3)=\cdots=g(v_{r-1}v_r),
	\end{equation} 
	which implies that
	\begin{equation}
	f(v_1)+f(v_2)=f(v_2)+f(v_3)=\cdots=f(v_{r-1})+f(v_r).
	\end{equation} 
	But this implies that
	\begin{equation}
	f(v_1)=f(v_3)=\cdots
	\end{equation} 
	and
	\begin{equation}
	 f(v_2)=f(v_4)=\cdots.
	\end{equation}
	So, $C$ is a $2$-colored cycle, which contradicts the acyclicity of the coloring $f$.
	\end{proof}
A famous result of Borodin \cite{Borodin} asserts that $\chi_a(G)\leqslant 5$ for every planar graph $G$, which is optimal in this class of graphs (see Fig. \ref{Zlamasy2}).
\begin{figure}[ht]
	
	\begin{center}
		
		\resizebox{13cm}{!}{
			
			\includegraphics{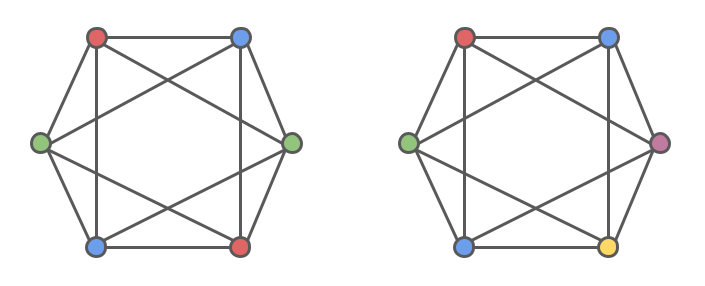}
			
		}
		
		\caption{A planar graph $G$ satisfying $\chi(G)=3$ and $\chi_a(G)=5$ (a proper 3-coloring (left) and an acyclic 5-coloring (right)).}
		\label{Zlamasy2}
	\end{center}
\end{figure}
\begin{corollary}\label{Corollary Acyclic Planar}
	Every planar graph $G$ satisfies $\zeta(G)\leqslant 5$.
\end{corollary}
We do not know if this bound is optimal. There exist planar graphs with $\zeta(G)=4$ (see Fig. \ref{Zlamasy1}), but we have not found one demanding the fifth color (see the final section for a discussion).

\begin{corollary}\label{Corollary Acyclic Outerplanar}
	Every outerplanar graph $G$ satisfies $\zeta(G)\leqslant 3$.
\end{corollary}
\begin{proof}
	It suffices to notice that $\chi_a(G)\leqslant 3$.
\end{proof}
Notice that this bound is optimal if only $G$ contains a triangle. In general, if a planar graph $G$ on $n$ vertices is triangle-free, then by Euler's formula it has at most $2n-4$ edges, so $\arb(G)\leqslant 2$. This leads to the following simple result.

\begin{theorem}\label{Theorem Triangle-free Planar}
	Every triangle-free planar graph $G$ satisfies $\zeta(G)\leqslant 4$.
\end{theorem}
\begin{proof}
First notice that any forest $F$ has a $2$-coloring of the edges such every path of length $3$ is non-monochromatic. Indeed, this can be done by coloring the edges alternately, accordingly to the parity of the distance to the root. Now, since $\arb(G)\leqslant 2$, $G$ has an edge decomposition into two forests, $F_1$ and $F_2$. We may color each of these forests by two disjoint pairs of colors, as above, and there is no monochromatic path of length $3$ in the whole graph $G$. Since there are no broken cycles of size two in $G$, the proof is complete.
\end{proof}

A natural intuition is that we should further go down with the number of colors if the girth of a graph is sufficiently large. Indeed, in \cite{BartnickiBCFGM} we proved that $\arb_2(G)\leqslant 3$ if the girth of a planar graph $G$ is at least $8$. Hence, for such graphs we have $\zeta(G)\leqslant 3$. The next result shows that we may attain the best possible bound for graphs of bounded genus if the girth is sufficiently large.

\begin{theorem}\label{Theorem Genus Girth}
	If $G$ is a planar graph of girth at least $13$, then $\zeta(G)\leqslant2$. More generally, for every genus $\gamma>0$ there exists $g(\gamma)$ such that every graph $G$ with genus $\gamma$ and girth $g(\gamma)$ satisfies $\zeta(G)\leqslant 2$.  
\end{theorem}
\begin{proof}
Let $G$ be a graph with genus $\gamma$. It is well known that if girth of $G$ is sufficiently large, then the vertices of $G$ can be split into two subsets, $V(G)=A\cup F$, such that $F$ induces a forest, while $A$ is $2$-independent, which means that every pair of vertices in $A$ is at distance at least $3$. In other words, the edges of $G$ between $A$ and $F$ form a star forest with star centers in $A$. Clearly, coloring the edges of the forest $G[F]$ by one color and the rest of edges (the star forest between $A$ and $F$) by the other color gives a coloring with no monochromatic broken cycle. The first part of the theorem follows from the result in \cite{BuCMRW}. 
\end{proof}
\subsection{Graphs of bounded arboricity}
In the next result we obtain a simple upper bound on $\zeta(G)$ in terms of the arboricity and the chromatic number. Since the later parameter is bounded in terms of the former, we get that $\zeta(G)$ is bounded for graphs of bounded $\arb(G)$.

\begin{theorem}\label{Theorem Arboricity Chromatic}
	Every graph $G$ satisfies $\zeta(G)\leqslant2\chi(G)\arb(G)$. Moreover, if $G$ is triangle-free, then $\zeta(G)\leqslant2\arb(G)$. 
\end{theorem}
\begin{proof}
	Assume that $\chi(G)=k$ and $\arb(G)=\ell$. 
	Let $f$ be a proper vertex coloring of $G$ by $k$ colors taken from the group $\mathbb{Z}_k$. Let $g$ be a derived coloring of the edges, defined as in the previous proof by $g(uv)=f(u)+f(v)$. Notice that in coloring $g$, every triangle is rainbow. Hence, there are no monochromatic broken cycles of size $2$ in coloring $g$.
	
	To handle longer broken cycles we use the arboricity. Let $h$ be any woody coloring of the edges of $G$ with $\ell$ colors. So, each color class is a forest and we may color every tree with two shades of the color of the forest it belongs to so that every path with at least three edges is not monochromatic. It follows that every cycle with at least $4$ edges is either $3$-colored or it is a $2$-colored $C_4$ with exactly two edges in each of the two colors. So, in coloring $h$ there are no monochromatic broken cycles of size at least $3$.
	
	To complete the proof of the first assertion it suffices to construct the product coloring $p$ of the edges of $G$ defined by $p(e)=(g(e),h(e))$. For the second assertion the coloring $h$ alone is sufficient.
	\end{proof}
\begin{corollary}\label{Corollary Arboricity Square}
Every graph $G$ satisfies $\zeta(G)\leqslant 4(\arb(G))^2$.
\end{corollary}
\begin{proof}
	Assume that $\arb(G)=\ell$. Then $G$ is $(2\ell-1)$-degenerate, which implies that $\chi(G)\leqslant 2\ell$. By the above theorem we get $\zeta(G)\leqslant 4\ell^2$. 
\end{proof}

\subsection{Graphs of bounded degree}
Let $G$ be a graph of maximum degree $\Delta$. By the Nash-Williams theorem \cite{Nash-Williams}, $\arb(G)\leqslant \left\lceil\frac{\Delta+1}{2}\right\rceil$. For instance, for a clique $K_{\Delta+1}$ we have $\arb(K_{\Delta+1})=\left\lceil\frac{\Delta+1}{2}\right\rceil$. On the other hand, any strongly woody coloring of the clique must be proper in the usual sense. Thus, we have $\zeta(K_{\Delta+1})=\chi'(K_{\Delta+1})=\Delta$ or $\Delta+1$, which is twice the arboricity of $K_{\Delta+1}$. We shall demonstrate, however, that for graphs of maximum degree $\Delta$ and sufficiently large girth the strong arboricity attains the minimum possible value and equals $\arb(G)$.

We will derive this fact as a consequence of the following result of Alon, Ding, Oporowski, and Vertigan \cite{AlonDOV}, concerning graph partitions into parts having small connected components. Indeed, suppose that in an edge colored graph $G$ by $k$ colors, the maximum size of any monochromatic connected subgraph is at most $C$. In particular, there is no monochromatic path of length $C+1$. Then, assuming that the girth of $G$ is at least $C+2$, we get that there is no monochromatic broken cycle in $G$, and therefore $\zeta(G)\leqslant k$.

\begin{theorem}[Alon, Ding, Oporowski, and Vertigan, \cite{AlonDOV}] Every graph of maximum degree $\Delta\geqslant 2$ has an edge $\left\lceil\frac{\Delta+1}{2}\right\rceil$-coloring such that every monochromatic component has at most $60\Delta-63$ edges.
\end{theorem}
Using this theorem we get immediately the following result.
\begin{corollary}\label{Corollary Maximum Degree}
	Every graph of maximum degree $\Delta\geqslant 2$ and girth at least $60\Delta - 61$ satisfies  $\zeta(G)=\arb(G)$.
\end{corollary}

Notice that in a special case of $\Delta=3$ we may get a better bound on the girth, by using a celebrated result of Thomassen \cite{Thomassen-Cubic}, asserting that every cubic graph has a $2$-edge-coloring in which every monochromatic component is a path with at most five edges.

\begin{corollary}\label{Corollary Thomassen Cubic}
	Every cubic graph $G$ of girth at least $7$ satisfies $\zeta(G)=2$.
\end{corollary}

\subsection{Minor-closed classes of graphs}It is clear that large girth alone is not sufficient for bounded arboricity. This follows from the celebrated result of Erd\H{o}s (see \cite{AlonSpencer}) establishing existenece of graphs with arbitrarily large girth and chromatic number. However, if we restrict to a proper minor-closed class of graphs, then the situation looks different.

\begin{theorem}[Thomassen \cite{Thomassen}] Let $\mathcal{M}$ be a proper minor-closed class of graphs. Then there exists a constant $g=g(\mathcal{M})$ such that every graph $G\in \mathcal{M}$ of girth at least $g$ is $2$-degenerate (in consequence, $\arb(G)\leqslant 2$).
\end{theorem}

By Theorem \ref{Theorem Arboricity Chromatic} it follows that graphs from proper minor-closed classes with sufficiently large girth satisfy $\zeta(G)\leqslant 4$. However, this bound does not seem optimal. It is perhaps true that two colors are sufficient for strongly woody coloring of $2$-degenerated graphs with sufficiently large girth.

\section{Final remarks and open problems}

We conclude this short note with a collection of open problems. The most intriguing is the question concerning the strong arboricity of planar graphs.
\begin{conjecture}
	Every planar graph $G$ satisfies $\zeta(G)\leqslant 4$.
\end{conjecture}
By the celebrated Four Color Theorem one gets easily a $3$-edge coloring of any planar graph $G$ with rainbow triangles. Indeed, one may start with a proper vertex coloring by the group $\mathbb{Z}_2\times\mathbb{Z}_2$ and then color every edge by the sum of colors on its ends. It seems plausible that including the fourth color one may also take care of longer broken cycles.

It is also natural to wonder about the best possible upper bound on $\zeta(G)$ in terms of the arboricity. As mentioned before, for every clique $K_n$ we have $\zeta(K_n)=\chi'(K_n)$, while $\arb(K_n)=\lceil n/2\rceil$. Thus, we have $\zeta(K_n)=2\arb(K_n)-1$. This suggests the following supposition.
\begin{conjecture}
	Every graph $G$ satisfies $\zeta(G)\leqslant 2\arb(G)$.
\end{conjecture}
It is known that every graph $G$ with $\arb(G)=k$ is at most $(2k-1)$-degenerated. Recall that the \emph{coloring number} $\col(G)$ is the least integer $r$ for which there is a vertex ordering with maximum back-degree equal to $r-1$. This leads to the following stronger conjecture.

\begin{conjecture}
	Every graph $G$ satisfies $\zeta(G)\leqslant \col(G)$.
\end{conjecture}
 
 Finally, we formulate a conjecture expressing a natural guess that large girth allows for strong arboricity to be as low as possible.

\begin{conjecture}
	For every integer $k$ there is an integer $g(k)$ such that every graph $G$ with $\arb(G)\leqslant k$ and girth at least $g(k)$ satisfies $\zeta(G)=\arb(G)$.
\end{conjecture}

Let us mention at the end that the notion of strong arboricity can be defined and studied for arbitrary \emph{matroids} in much the same way as it is with the usual arboricity. Indeed, broken cycles can be defined \emph{mutatis mutandis} in the matroid setting. It seems plausible that the corresponding parameter $\zeta(M)$ is bounded in terms of $\arb(M)$ for an arbitrary matroid $M$.

\end{document}